\documentclass[12pt, a4paper, reqno]{amsart}

\usepackage[active]{srcltx}

\usepackage{graphicx}
\usepackage{color}
\usepackage{enumerate}
\usepackage{amssymb}
\usepackage{overpic}
\newtheorem{lemma}{Lemma}[section]
\newtheorem{corollary}[lemma]{Corollary}
\newtheorem{theorem}[lemma]{Theorem}
\newtheorem{proposition}[lemma]{Proposition}

\newcommand{\R}{\mathbb{R}}
\newcommand{\C}{\mathbb{C}}

\newcommand{\N}{\mathbb{N}}

\author[A. Gasull]{Armengol Gasull}
\address{Dept. de Matematiques.
Universitat Autonoma de Barcelona. Edifici C. 08193 Bellaterra,
Barcelona. Spain} \email{gasull@mat.uab.cat}

\author[H. Giacomini]{Hector Giacomini}
\address{Laboratoire de Math\'{e}matiques et Physique Th\'{e}orique.
 Facult\'{e} des  Sciences et Techniques. Universit\'{e} de Tours, C.N.R.S. UMR 7350.
37200 Tours. France} \email{Hector.Giacomini@lmpt.univ-tours.fr}

\date{}
\dedicatory{} \commby{}

\begin{document}

\title
[Explicit traveling waves and invariant algebraic curves]{Explicit
traveling waves\\ and invariant algebraic curves}
\begin{abstract}
In this paper we introduce a precise definition of algebraic
traveling wave solution for general $n$-th order partial
differential equations. All examples of explicit traveling waves
known by the authors fall in this category. Our main result proves
that algebraic traveling waves exist if and only if  an associated
$n$-dimensional first order ordinary differential system has some
invariant algebraic curve. As a paradigmatic application we prove
that, for the celebrated Fisher-Kolmogorov equation, the only
algebraic traveling waves solutions are the ones found in 1979 by
Ablowitz and Zeppetella.  To the best of our knowledge, this is the
first time that this type of results have been obtained.
\end{abstract}

\maketitle

\section{Introduction and Main Results}

Mathematical modelling of dynamical processes in a great variety of
natural phenomena leads in general to non-linear partial
differential equations. There is a particular class of solutions for
these non-linear equations that are of considerable interest. They
are the traveling wave solutions \cite{GilKers,GriSchi,Gri,RodMiu}.
Such a wave  is a special solution of the governing equations, that
may be localised or periodic, which does not change its shape and
which propagates at constant speed.  In the case of linear equations
the profile is usually arbitrary.  In contrast, a non-linear
equation will normally determine a restricted class of profiles, as
the result of a balance between nonlinearity and dissipation. These
waves appear in fluid dynamics \cite{Johnn,Kund}, chemical kinetics
involving reactions \cite{GilKers,Lieh}, mathematical biology
\cite{Gri,Mur}, lattice vibrations in solid state physics
\cite{Maf},
 plasma physics and laser theory \cite{InfRow}, optical fibers \cite{AA}, etc.
In these systems the phenomena of dispersion, dissipation,
diffusion, reaction and convection are the fundamental physical
common facts.

There is an increasing interest in finding explicit exact solutions
for these traveling waves. There are several standard methods for
obtaining such solutions, as the inverse scattering transformation
\cite{AblCla,DraJoh}, the Backlund transformation
\cite{AblCla,DraJoh},  the Painlev\'{e} method \cite{Gor} and the
Hirota's bilinear method \cite{Johnn}.

The inverse scattering transformation is a non-linear analog of the
Fourier transform used for solving linear equations. This method
allows certain non-linear problems, called integrable, to be treated
by what are essentially linear methods.

The Backlund transformation allows to find solutions to a non-linear
partial differential equation from either a known solution to the
same equation or from a solution to another equation. This can
enable  one to find more complex solutions from a simple one, e.g. a
multi-soliton solution from a single soliton solution.

The Painlev\'{e} method is a procedure to detect integrable differential
equations. The Lie group method is applied to a partial differential
equation for finding group-invariant solutions that satisfy ordinary
differential equations. Then the Painlev\'{e} property is tested for
these reduced equations. An ordinary differential equation is said
to have the Painlev\'{e} property if the general solution has no movable
critical singularities. Movable refers to the arbitrary position of
the solution's singularities in complex time. For any solution the
presence and position of movable singularities is given by the
initial conditions. The other type of singularities that can be
found are fixed singularities.

The Hirota's direct method is employed for constructing
multi-soliton solutions to integrable non-linear evolution
equations. The method is based  on introducing a transformation into
new variables, so that in these new variables multi-soliton
solutions appear in a particularly simple form. In fact they appear
as polynomials of simple exponentials in the new variables. This
transformation requires sometimes the introduction of new dependent
and sometimes even independent variables. Expressed in the new
variables the equation will be quadratic in the dependent variables
(the so-called Hirota's bilinear form) and the derivatives must only
appear in combinations that can be expressed using Hirota's
differential operator.

We consider in this work general $n$-th order partial differential equations of
the form
\begin{equation}\label{eq:0}
\frac{\partial^n u}{\partial x^n}=F\Big(u, \frac{\partial
u}{\partial x},\frac{\partial u}{\partial t},\frac{\partial^2
u}{\partial x^2},\frac{\partial^2 u}{\partial x\partial
t},\frac{\partial^2 u }{\partial t^2},\ldots,\frac{\partial^{n-1}u
}{\partial x^{n-1}},\frac{\partial^{n-1} u }{\partial
x^{n-2}\partial t}, \ldots, \frac{\partial^{n-1} u}{\partial
x\partial t^{n-2}},\frac{\partial^{n-1} u}{\partial t^{n-1}}\Big),
\end{equation}
where $x$ and $t$ are real variables and $F$ is a smooth map. The
traveling wave solutions (TWS) of~\eqref{eq:0} are particular
solutions of the form $u=u(x,t)=U(x-ct)$,  where $U(s)$ satisfies
the boundary conditions
\begin{equation}\label{eq:ab}
\lim_{s\to-\infty} U(s)=a\quad\mbox{and}\quad \lim_{s\to\infty}
U(s)=b,
\end{equation}
 where $a$ and $b$ are solutions, not necessarily different, of
$F(u,0,\ldots,0)=0$.
 Plugging $u(x,t)=U(x-ct)$  into~\eqref{eq:0} we get that $U(s)$ has
to be a solution, defined for all $s\in\R$,  of the $n$-th
order ordinary differential equation
\begin{multline}\label{edon}
U^{(n)}=F\big(U,U',-cU',U'',-cU'',c^2U'',\ldots,\\U^{(n-1)},-cU^{(n-1)},\ldots
,(-c)^{n-2}U^{(n-1)}, (-c)^{n-1}U^{(n-1)}\big),
\end{multline}
 where $U=U(s)$ and the derivatives  are taken with
respect to $s$. The  parameter $c$ is called the {\it speed} of the
TWS.

We remark that although in this paper we restrict our attention to
TWS associated to only one partial differential equation and
$x\in\R$, our approach can be extended to systems of partial
differential equations, with ${\bf u}\in\R^d$ and ${\bf x}\in\R^m$.
In this situation,  we would search for TWS of the form $u_j({\bf
x},t)=U_j({\bf k}\cdot{\bf x}-ct)$, $j=1,\ldots, d$, for some  ${\bf
k}\in\R^m$ and  $c\in\R.$

\vspace{0.2cm}

\noindent{\bf Definition.} We will say that $u(x,t)=U(x-ct)$ is an
{\it algebraic TWS}\, if $U(s)$ is a non constant function that
satisfies \eqref{eq:ab} and \eqref{edon} and   there exists a
polynomial $p\in\R[z,w]$ such that $p(U(s),U'(s))=0$.

\vspace{0.2cm}

All the explicit TWS known by the authors are algebraic when $F$ is
a polynomial.  Let us present several well-known examples.

 Consider at first the Burgers equation
\begin{equation*}
\frac{\partial u}{\partial t}+u\frac{\partial u}{\partial x}-a\frac{\partial^2 u}{\partial x^2}=0,
\end{equation*}
where $a\ne0$ is an arbitrary constant. This equation appears in the
 modeling of acoustic and hydrodynamic waves, gas dynamics and
traffic flow (see \cite{PoZai}) and has the one-parametric family of
solutions
\begin{equation*}
u(x,t)=c\Big(1-\tanh\big(\frac{c}{2a} (x-ct)\big)\Big),
\end{equation*}
where $c$, the speed of the wave, is an arbitrary constant. For this
case  $p(U,U')=2aU'+(2c-U)U.$

 The famous Korteweg-de Vries
equation
\begin{equation*}
 \frac{\partial u}{\partial t}-6u\frac{\partial u}{\partial x}+\frac{\partial^3 u}{\partial x^3}=0
\end{equation*}
appears in several domains of physics, non-linear mechanics, water
waves, etc (see \cite{AblCla,DraJoh,knob,PoZai}). It has a
one-parametric family of solutions given by
\begin{equation*}
 u(x,t)=\frac{-c}{2\cosh^2(\frac{\sqrt{c}}{2}(x-ct))},
\end{equation*}
where $c$ is an arbitrary positive parameter. For this second
example  $p(U,U')=(U')^2-(c+2U)U^2.$

 Consider
now the Boussinesq equation
\begin{equation*}
 \frac{\partial^2 u}{\partial t^2}+u\frac{\partial^2 u}{\partial x^2}
 -\frac{\partial^2 u}{\partial x^2}+\Big(\frac{\partial u}{\partial x}\Big)^2
-\frac{\partial^4 u}{\partial x^4}=0.
\end{equation*}
This equation describes surface water waves (see \cite{AblCla,Johnn}) and has the two-parametric family of solutions
\begin{equation*}
 u(x,t)=(1-8k^2-c^2)+12k^2\tanh^2(k(x-ct)),
\end{equation*}
where $k$ and $c$ are arbitrary constants. Here we have
\begin{align*}
p(U,U')=& 3(U')^2 -{U}^{3}-3  \left( c-1 \right)  \left( 1+c \right)
{U}^{2}-3
 \left( {c}^{2}-1+4 {k}^{2} \right)  \left( {c}^{2}-1-4 {k}^{2}
 \right) U\\&- \left( {c}^{2}-1+8 {k}^{2} \right)  \left( {c}^{2}-1-4 {
k}^{2} \right) ^{2}.
\end{align*}

We consider now the so-called improved modified Boussinesq equation
\begin{equation*}
 \frac{\partial^2 u}{\partial t^2}-u\frac{\partial^2 u}{\partial x^2}-
 \frac{\partial^2 u}{\partial x^2}-\Big(\frac{\partial u}{\partial x}\Big)^2
-\frac{\partial^4 u}{\partial x^2 \partial t^2}=0.
\end{equation*}
This equation appears in the modeling of non-linear waves in a
weakly dispersive medium
 (see for instance \cite{KanNish}) and has a three-parametric family of  TWS given by
\begin{equation*}
 u(x,t)=c^2-1+4c^2k^2-8c^2mk^2+12c^2mk^2\operatorname{cn}^2(k(x-ct),m),
\end{equation*}
where $c$, $k$ and $m$ are arbitrary constants and
$\operatorname{cn}(x,m)$ is the Jacobi elliptic function of
 elliptic modulus $m$ that reduces to $\cos(x)$ when $m=0$.
  In this equation, this family of traveling waves and many others
have been found in \cite{ZonHon}.  For this case
\begin{align*}
p(U,U')=& 3c^2(U')^2 +{U}^{3}+3  \left( 1-c^2 \right)  {U}^{2}\\&+
\left( 48{c}^{4} \left(m- {m}^{2}-1 \right) {k}^{4}+3\left(1- c^2
 \right) ^{2} \right) U\\&+64{c}^{6} \left( -1+2m \right)  \left( m+1 \right)  \left( m-2
 \right) {k}^{6}\\&+48{c}^{4}  \left(
 1-c^2\right)\left(m- {m}^{2}-1\right)  {k}^{4}+ \left( 1-c^2 \right)
 ^{3}.
\end{align*}

Notice also that the class of  TWS given by $U(s)=q(e^{\lambda s})$
for some real number $\lambda\ne0$ and some rational function
$q\in\R(z)$, that are usually obtained with the so-called
exp-function method (\cite{J}), are always algebraic TWS. In this
case $U'(s)=\lambda q'(e^{\lambda s})$. Write
$U(s)={q_1(z)}/{q_2(z)},$ and $U'(s)={q_3(z)}/{q_4(z)},$ with
$z=e^{\lambda s},$ for some polynomials $q_j\in\R[z],j=1,\ldots,4.$
Then,  define
\[
p(U,U')=\mbox{Res}\big(q_2(z)U-q_1(z),q_4(z)U'-q_3(z),z\big),
\]
where  $\mbox{Res}(M(z),N(z),z)$ denotes the resultant of the
polynomials $M$ and $N$ with respect to $z;$ see~\cite[p.45]{St}.
Then, clearly $p(U(s),U'(s))=0$ for some polynomial $p$, as we
wanted to see.

It is known that the  TWS  correspond to homoclinic ($a=b$) or
heteroclinic ($a\ne b$) solutions of an associated $n$-dimensional
system of ordinary differential equations, see also the proof of
Theorem~\ref{t0:main}. In many cases,  the critical points where
these invariant manifolds start and end are hyperbolic. When $F$ is
regular  we get, using for instance normal form theory, that in a
neighborhood of each of these points, this manifold can be
parameterized as $\varphi(e^{\lambda s})$, for some smooth function
$\varphi$,  where $\lambda$ is one of the eigenvalues of the
critical points.  This fact, together with the above list of
examples, motivate our definition of algebraic~TWS.

Our main result, which is  proved in Section~\ref{proof}, is:

\begin{theorem}\label{t0:main} The partial differential
equation~\eqref{eq:0} has an algebraic  traveling wave solution with
speed $c$ if and only if the first order  differential system

\begin{equation}\label{system}
\left\{
\begin{array}{ccl}
y_1'&=&y_2, \\
y_2'&=&y_3,\\
\vdots&&\vdots\\
y_{n-1}'&=&y_{n},\\
y_n'&=& G_c(y_1,y_2,\ldots,y_{n}),
\end{array}
\right.
\end{equation}
where
\begin{multline*}
G_c(y_1,y_2,\ldots,y_{n})=F(y_1,y_2,-cy_2,y_3,-cy_3,c^2y_3,\ldots,\\y_{n},-cy_{n},\ldots
,(-c)^{n-2}y_{n},(-c)^{n-1}y_{n}), \end{multline*} has an invariant
algebraic curve containing the critical points $(a,0,\ldots,0)$ and
$(b,0,\ldots,0)$ and no other critical points between them.
\end{theorem}

Recall that, as usual, we will say that a differential system has an
invariant algebraic curve $C$ if this curve is invariant by the flow
and moreover it is contained in the intersection of $n-1$
functionally independent  algebraic varieties of codimension one. We
remark that these varieties do not need to be necessarily invariant
by the flow of the system.

 When $F$  is a polynomial, the condition for the
existence of an algebraic TWS is that a certain polynomial
differential system must have an algebraic invariant curve. The
problem of determining necessary conditions for the existence of
algebraic invariant curves for polynomial differential systems goes
back to the work of Poincar\'{e}. This problem   has been extensively
investigated in the last years for the case $n=2$, see for instance
\cite{maite,DuLlAr,llibre} and references therein, but for $n>2$ the
research is only beginning, see for instance \cite{GGG,llibrezhang}.
As a consequence, for second order partial differential equations of
the form~\eqref{eq:0}, our result translates the question of the
existence of algebraic TWS to a related problem for which many tools
are available.

We remark that explicit TWS have also been searched   for by using
several direct methods, such as the exp-function method and the
tanh-function method and its variants, see for instance~\cite{J,
Maf, MafHer1, MafHer2, ZonHon}. These methods are essentially based
on the following idea: fix a class of functions with several free
parameters and then impose conditions on the parameters to find some
particular cases satisfying the corresponding equations. For
instance, the four examples of algebraic TWS given above can be
obtained by applying these direct methods.

On the contrary, our approach gives necessary and sufficient
conditions for a partial differential equation to have explicit
algebraic TWS. To the best of our knowledge, this is the first time
that this type of results have been obtained. As a paradigmatic
example, we  apply our method to the celebrated Fisher-Kolmogorov
reaction-diffusion partial differential equation
\begin{equation}\label{eq:1}
\frac{\partial u}{\partial t}=\frac{\partial^2 u}{\partial x^2}+
u\,(1-u),
\end{equation}
introduced in 1937 in the classical papers~\cite{Fisher, Kolmo}  to
model the spreading of biological populations;  see also \cite{GGT}
for some recent results. For this equation $a=1$ and $b=0$
 in~\eqref{eq:ab}. Moreover, from~\cite{Fisher,Kolmo}, it is also
known that the traveling waves only exist for $c\ge2.$ We prove:

\begin{theorem}\label{t:main} The Fisher-Kolmogorov
equation~\eqref{eq:1} has algebraic traveling wave solutions only
when the speed is $c=5/\sqrt{6}$ and  they are the ones given by
Ablowitz and Zeppetella (\,\cite{AZ}):
\[
u(x,t)=\frac{1}{\left(1+ k e^{\frac{1}{\sqrt
6}\,\left(x-\frac5{\sqrt{6}}t\right) }\right)^2},\quad k>0.
\]
\end{theorem}

These explicit TWS have been found by applying the Painlev\'{e} method;
 see \cite{Gor} for an introduction to this method.

 Notice that
for~\eqref{eq:1},  the above function is   an algebraic TWS, because
the corresponding $U(s)$ satisfies
\[
p(U,U')=3(U')^2+2\sqrt{6}U U' +2(1-U)U^2=0.
\]
 We remark that
this family of TWS only exists for a fixed value of the speed $c$,
while for the other examples given above
 the speed $c$ is arbitrary.  This can also  be seen
 in the corresponding  associated systems~\eqref{system}, because in
 all these cases, for all values of $c$, the system possesses an
 invariant algebraic curve. In fact, in the first two  equations
 (Burgers and Korteweg-de Vries) all the solutions of the vector
 fields are contained in algebraic curves.

 Another family exhibiting
 algebraic TWS for a given speed $c$ is presented in Section
 \ref{se:3}. It includes the so-called  Nagumo equation; see~\cite{Mur1}.

 Our approach can be applied to characterize the existence of
algebraic TWS for many other polynomial partial differential
equations, like for instance the Newell-Whitehead-Segel
equation(\cite{NewWhi,seg}), the Zeldovich equation(\cite{ZelFra})
or  some of the equations considered
in~\cite{Gor,Gri,SanMai1995,SanMai1997,Xin}.

\section{Proof of Theorem~\ref{t0:main}}\label{proof}

In this section we prove Theorem~\ref{t0:main} and   give some of
its consequences.   Furthermore,   we introduce an algebraic
characterization of the  planar invariant algebraic curves.

\begin{proof}[Proof of Theorem~\ref{t0:main}]

Assume first that  the partial differential equation~\eqref{eq:0}
has an algebraic TWS, $u(x,t)=U(x-ct)$, with $p(U(s),U'(s))=0$ for
some polynomial $p$. For the sake of notation we define $p_1:=p$ and
\[
p_2(U(s),U'(s),U''(s)):=\mathcal{D}_1\,p_1(U(s),U'(s))\,U'(s)+\mathcal{D}_2\,p_1(U(s),U'(s))\,U''(s),\]
where $\mathcal{D}_1$ and $\mathcal{D}_2$ indicate
 partial derivatives with respect to the first and second variables of
 $p_1(U,U')$, respectively, and $p_2\in\R[u,v,w]$. Notice that  since
 $p_1(U(s),U'(s))=0$ it holds that
$p_2(U(s),U'(s),U''(s))=0$.
 Doing successive  derivatives we
 obtain $n-3$ new polynomials $p_j$, $j=3,\ldots,n-1$, for which
 \[
p_j(U(s),U'(s),U''(s),\ldots, U^{(j)}(s))=0.
 \]
Using all the above equalities, and  the fact that $U$ gives a TWS,
we obtain that the vector function
\[(y_1(s),y_2(s),\ldots,y_n(s))=\big( U(s),U'(s),\ldots, U^{(n-1)}(s)\big)\]
is a parametric representation of a curve $\mathcal C$ in the phase
space of the system~\eqref{system} associated to~\eqref{eq:0}. In
fact, $C$ is an algebraic curve, because it is contained in the
intersection of the $n-1$ functionally independent algebraic
hypersurfaces $p_j(y_1,\ldots,y_j)=0, j=1,2\ldots,n-1,$ that is,
\[
\mathcal{C}\subset\bigcap_{j=1}^{n-1}\,\big\{p_j(y_1,\ldots,y_j)=0\big\}.
\]
As $U$ satisfies~\eqref{eq:ab} the   system has no critical points
on this curve between $(a,0,\ldots,0)$ and $(b,0,\ldots,0)$. Hence
the first part of the theorem follows.

Assume, to prove the converse implication,  that
system~\eqref{system} has an algebraic invariant curve. Let
\[
{\bf y}(s)=\big(U(s),U'(s),U''(s),\ldots, U^{(n-1)}(s)\big)
\]
be the solution of system~\eqref{system} associated to this curve
and joining the critical points  $(a,0,\ldots,0)$ and
$(b,0,\ldots,0)$.
 By definition, this curve  is included in the intersection of
$n-1$ codimension one functionally independent algebraic
hypersufaces $q_j(y_1,y_2,\ldots,y_n)=0$, $j=1,2,\ldots,n-1.$
Therefore,  $U(s)$ must satisfy the $n-1$ polynomial differential
equations
\[
q_j(U(s),U'(s),\ldots, U^{(n-1)}(s))=0,\quad j=1,2,\ldots, n-1.
\]
Doing successive resultants, we obtain that $U$ satisfies all the
resulting lower order polynomial differential equations. This
procedure  arrives to a polynomial first order equation
$q(U(s),U'(s))=0$. This equation proves that the TWS is algebraic.
\end{proof}

 In view of  our  result, we  need a method to detect when a
polynomial system of ordinary differential equations has algebraic
invariant curves to determine whether  some polynomial partial
differential equation can have algebraic TWS.

Although, as we have already explained in the introduction, there
are some works dealing with this problem in the $n$-dimensional
setting~\cite{GGG,llibrezhang}, the planar case  is the most
developed one.

Consider a planar differential system,
\begin{equation}\label{gene}
\left\{
\begin{array}{ccl}
 x'&=&P(x,y), \\
y'&=&Q(x,y),
\end{array}
\right.
\end{equation}
where $P$ and $Q$ are polynomials of degree at most~$N$, and assume
that there is a polynomial $g(x,y)$ such that the set $\{g(x,y)=0\}$
is non-empty and invariant by the flow of~\eqref{gene}. If $g$ is not
irreducible in $\C[x,y]$ then there exist several irreducible
polynomials, $\widetilde g_j, j=1,\ldots,k$, such that for each
$j$, the corresponding set $\{\widetilde g_j(x,y)=0\}$ is also
non-empty and invariant by the flow of the system and
$\{g(x,y)=0\}=\cup_{j=1}^k\{\widetilde g_j(x,y)=0\}$.

For irreducible polynomials we have the following algebraic
characterization of invariant algebraic  curves, which is the one
that we will use in Section~\ref{qqss}. Given an irreducible
polynomial of degree $n$, $f(x,y)$, then $f(x,y)=0$ is an {\it
invariant algebraic curve} for the system if there exists a
polynomial of degree at most $N-1$, $k(x,y)$, called the {\it
cofactor of $f$,} such that
\begin{equation}\label{condi}
P(x,y)\dfrac{\partial f(x,y)}{\partial x}+Q(x,y)\dfrac{\partial
f(x,y)}{\partial y}-k(x,y)f(x,y)=0.
\end{equation}
For a proof of this result see \cite{maite,DuLlAr,llibre}. The above
characterization is also used
 for $n$-dimensional systems to  determine codimension one invariant
 algebraic varieties;  see for instance~\cite{llibrezhang}.
 For finding invariant algebraic curves  the
  cofactor is then exchanged for a $(n-1)\times(n-1)$ matrix of cofactors,
 see~\cite{GGG}.

\section{Proof of Theorem~\ref{t:main} and other
examples}\label{se:3}

Our  proof of Theorem~\ref{t:main} is based on the following result,
which will be proved in  the next section.

\begin{theorem}\label{teo:iac} Consider  the system
\begin{equation}\label{main2}
\left\{
\begin{array}{ccl}
 x'&=&-y, \\
y'&=&-x-c y +x^2,
\end{array}
\right.
\end{equation}
with $c\ge2.$ Assume that  it has an irreducible invariant algebraic
curve that passes through the origin. Then $c=5/\sqrt{6}$ and this
curve is
\begin{equation}\label{iac}
y^2+2\sqrt{\frac23}(1-x)y+\frac23x(1-x)^2=0.
\end{equation}
\end{theorem}

\begin{proof}[Proof of Theorem~\ref{t:main}]
Assume that ~\eqref{eq:1} has an algebraic TWS,
$u(x,t)=U(x-ct)$. By the results of \cite{Fisher,Kolmo} we already
know that $c\ge2.$ Moreover, by Theorem~\ref{t0:main}, the planar
system
\begin{equation*}
\left\{
\begin{array}{ccl}
 y_1'&=&y_2, \\
y_2'&=&-cy_2-y_1(1-y_1),
\end{array}
\right.
\end{equation*}
should have an invariant algebraic curve, $g(y_1,y_2)=0$, containing
the critical points  $(0,0)$ and $(1,0)$.  Moreover, without loss of
generality, we can assume that it is irreducible.

Taking $x=1-y_1$ and $y=y_2$  we obtain  system~\eqref{main2}. Then,
it should also have an irreducible invariant algebraic curve
$f(x,y)=0$, with $f(0,0)=f(1,0)=0.$
By Theorem~\ref{teo:iac} we get that $c=5/\sqrt6$ and $f$ has to be
\begin{equation*}
f(x,y)=y^2+2\sqrt{\frac23}(1-x)y+\frac23x(1-x)^2.
\end{equation*}
The branch of $f(x,y)=0$ that contains the origin is
\[
y=A(1-\sqrt{1-x})(x-1),
\]
where $A:=\sqrt{6}/3.$ Using the first equation of~\eqref{main2},
that in this case is $x'=-y$, we obtain that
\[
x'(s)=A(1-\sqrt{1-x(s)})(1-x(s)).
\]
Returning to the function $U(s)=y_1(s)=1-x(s)$ we get the
differential equation
\[
U'(s)=-A\big(1-\sqrt{U(s)}\big)U(s).
\]
Introducing $W(s)=\sqrt{U(s)}$ we obtain that $W$ satisfies the
logistic differential equation
\[
W'(s)=-\frac A 2\big(1-W(s)\big)W(s).
\]
Its non-constant solutions that are  defined for all $s\in\R$  are
\[
W(s)=\frac 1{1+k e^{\frac A 2s}},\quad k>0.
\]
Hence
\[
U(s)=\frac 1{(1+k e^{\frac{A}2s})^2}=\frac 1{(1+k
e^{\frac{1}{\sqrt{6}}s})^2}
\]
and
\[
u(x,t)=\frac{1}{\left(1+ k e^{\frac{1}{\sqrt
6}\,\left(x-\frac5{\sqrt{6}}t\right) }\right)^2}, \quad k>0,
\]
as we wanted to prove.
\end{proof}

\subsection{A simple family with algebraic TWS}

 In this subsection we consider the family of second order
reaction-diffusion equations
\begin{equation}\label{edp-new}
\frac{\partial u}{\partial t}=-d\, f(u)(f'(u)+r)+d\frac{\partial^2
u}{\partial x^2},
\end{equation}
where $f$ is a polynomial function and $d>0$ and $r$ are real
constants. As we will see, studying its algebraic TWS we recover
some of the results presented in \cite[Ch.11]{Mur1}. In particular
we will find some algebraic TWS for the Nagumo equation, related
with the FitzHugh-Nagumo model for the nerve action potentials.

The  planar system \eqref{system} associated to \eqref{edp-new} is
\begin{equation}\label{pla}
\left\{
\begin{array}{ccl}
 x'&=&y, \\
y'&=&-\dfrac c d y + f(x)(f'(x)+r).
\end{array}
\right.
\end{equation}
It is easy to obtain one invariant algebraic curve for it for some
particular values of the parameters.

\begin{lemma} When $r=c/d$, system \eqref{pla} has the
invariant algebraic curve $y-f(x)=0.$
\end{lemma}
\begin{proof}
Consider the algebraic curve $H(x,y)=y-f(x)$ and $r=c/d.$ Then
\begin{align*}
y\dfrac{\partial H(x,y)}{\partial x}&+\left(-\dfrac c d y +
f(x)\big(f'(x)+\dfrac c d\big)\right)\dfrac{\partial H(x,y)}{\partial y}\\
=& -yf'(x)+\left(-\dfrac c d y + f(x)\big(f'(x)+\dfrac cd
\big)\right)\\=& -\left(f'(x)+\dfrac
cd\right)(y-f(x))=-\left(f'(x)+\dfrac cd\right)H(x,y).
\end{align*}
Hence the result follows.
\end{proof}

As a corollary of this  lemma and the results of the previous
section we have:
\begin{corollary}\label{corol} The solutions of the polynomial ordinary differential
equation $p(U(s),U'(s))=U'(s)-f(U(s))=0$, with adequate boundary
conditions, give the  algebraic TWS  of equation~\eqref{edp-new},
 $u(x,t)=U(x-ct)$ with speed $c=rd$.
\end{corollary}

Let us  apply this corollary to find algebraic TWS for the partial
differential equations:
\begin{align}
\frac{\partial u}{\partial
t}&=a(u-u_1)(u_2-u)(u-u_3)+d\frac{\partial^2 u}{\partial
x^2},\label{edp1}
\\
\frac{\partial u}{\partial t}&=u^{q+1}(1-u^q)+\frac{\partial^2
u}{\partial x^2},\label{edp2}
\end{align}
 where $a>0, d>0$, $u_1<u_2<u_3$ are given real constants and
$q\in\N^+.$

 Equation~\eqref{edp1}  is the Nagumo equation. Notice that it is
of the type~\eqref{edp-new} since the following equality holds
\begin{equation*}
a(u-u_1)(u_2-u)(u-u_3)=-d f(u)\Big(
f'(u)+\sqrt{\frac{a}{2d}}(u_1-2u_2+u_3)\Big),
\end{equation*}
where $f(u)=\sqrt{\frac{a}{2d}}(u-u_1)(u-u_3).$ Hence, using
Corollary \ref{corol}, we obtain that, when
\[
c=\sqrt{\frac{ad}2}(u_1-2u_2+u_3),
\]
equation \eqref{edp1} has the algebraic TWS,  $u(x,t)=U(x-ct)$,
where $U$ satisfies
\[
U'(s)=\sqrt{\frac{a}{2d}}(U(s)-u_1)(U(s)-u_3),
\]
which is a logistic equation. Its non-constant solutions that are
defined for all~$s$ are
\[
U(s)=\frac{u_3+k u_1 e^{\alpha(u_3-u_1)
s}}{1+ke^{\alpha(u_3-u_1)s}},\quad\mbox{with}\quad k>0 \quad
\mbox{and}\quad \alpha=\sqrt{\frac{a}{2d}}.
\]

 Similarly, we have the equality
\[
u^{q+1}(1-u^q)=-f(u)\Big(f'(u)+\frac1{\sqrt{q+1}}\Big),
\]
where $f(u)=\frac1{\sqrt{q+1}}u(u^q-1)$.  Applying again Corollary
\ref{corol},  with $d=1$, we obtain that when
\[
c=\frac1{\sqrt{q+1}},
\]  equation \eqref{edp2} has the algebraic TWS,  $u(x,t)=U(x-ct)$, where $U$ satisfies
\[
U'(s)=\frac1{\sqrt{q+1}}U(s)(U^q(s)-1).
\]
Its non-constant solutions that are defined for all $s$ are
\[
U(s)=\left(1+ k e^{\frac{q}{\sqrt{q+1}}s} \right)^{-\frac
1q},\quad\mbox{with}\quad k>0.
\]

We remark that studying all the invariant algebraic curves of the
planar system~\eqref{pla} we could know whether the corresponding
partial differential equation~\eqref{edp-new}  does or does not have
algebraic TWS with speed different from $rd.$

\section{Algebraic invariant curves for
system~\eqref{main2}}\label{qqss}

This section is devoted to  the proof of Theorem~\ref{teo:iac}. We
need some preliminary results. The first one collects some
well-known properties of the Gamma function, and also relates it
with the Pochhammer symbol, $x^{[m]}:=x(x+1)(x+2)\cdots (x+m-1).$

\begin{lemma}\label{gamma} For $x,y\in\R$ and $p,q,m\in\N$,
\begin{enumerate}[(i)]
\item  $\Gamma(x+1)=x\Gamma(x),$

\item $\prod_{j=p}^q(x+j)=\dfrac{\Gamma(x+q+1)}{\Gamma(x+p)},$

\item $\sum _{j=0}^{m}{ m\choose j}\Gamma  \left(x+ j\right) \Gamma \left(y+ m
-j \right)={\dfrac {\Gamma  \left( x \right) \Gamma  \left( y
\right) \Gamma
 \left( x+y+m \right) }{\Gamma  \left( x+y \right) }},$

\item $\sum _{j=0}^{m}{m\choose j}(m-j)\Gamma  \left(x+ j \right) \Gamma  \left(y+ m
-j \right)={\dfrac {m y\Gamma  \left( x \right) \Gamma  \left( y
\right) \Gamma
 \left(x+y+m \right) }{(x+y)\Gamma  \left( x+y\right) }},$

\item $\dfrac{\Gamma(x+m)}{\Gamma(x)}=x^{[m]}.$

\end{enumerate}

\end{lemma}

 The next results  reduce the set of possible invariant curves and
cofactors.

\begin{proposition}\label{pro:odd} If the quadratic system \eqref{main2} has an
irreducible invariant algebraic curve of degree $n$,  then its
cofactor $k(x,y)$ must be constant, i.e. $k(x,y)\equiv c_0$, and its
degree  has to be even.
\end{proposition}

\begin{proof} Since the system \eqref{main2} is quadratic ($N=2$),
the cofactor of an invariant algebraic curve of degree $n$,
$f_n(x,y)=0$, with
\begin{equation}\label{fn}
f_n(x,y)=h_n(x)y^n+h_{n-1}(x)y^{n-1}+\cdots+h_1(x)y+h_0(x),
\end{equation}
is linear, i.e $k(x,y)=c_0+c_1x+c_2y.$ Then, equation \eqref{condi}
writes as
\begin{equation}\label{condi2}
-y\dfrac{\partial f_n(x,y)}{\partial x}+(-x-cy+x^2)\dfrac{\partial
f_n(x,y)}{\partial y}-(c_0+c_1x+c_2y)f_n(x,y)=0.
\end{equation}

Imposing that the higher order  term in $y$ of the above equation
vanishes we get the differential equation
\[
c_2 h_n(x)+h'_n(x)=0.
\]
Since $h_n$ has to be a polynomial we obtain that $c_2=0$ and that
$h_n(x)$ is a constant. Hence, without loss of generality, we can
assume that $h_n(x)\equiv1.$ Then, equality \eqref{condi2} is
equivalent to the following set of linear differential equations
\begin{equation}\label{edos}
h'_{j-2}(x)=jx(x-1)h_j(x)-((j-1)c+c_0+c_1x)h_{j-1}(x), \quad
j=n+1,n,\ldots,2,1,
\end{equation}
where $h_n(x)\equiv 1$ and $h_{n+1}(x)\equiv h_{-1}(x)\equiv0.$

If $c_1\ne0$, using \eqref{edos} we can obtain the degrees of the
functions $h_j.$ They are:
\[\deg(h_{n-k})=2k,\quad k=0,1,\ldots,n-1,n.\]
In particular $\deg(h_1)=2n-2$ and $\deg(h_0)=2n.$ From
\eqref{edos}, for $j=1$, we obtain that
\begin{equation}\label{edo0}
-c_0h_0(x)-c_1x h_0(x)-xh_1(x)+x^2h_1(x)=0.
\end{equation}
Studying the higher order terms in $x$ of this equation we get that
 relation~\eqref{edo0}  can never be satisfied. As a consequence
$c_1=0$ and so $k(x,y)=c_0$ as we wanted to prove.

Consider now equation  \eqref{edos} with $c_1=0$. Assume, to arrive
to a contradiction, that $n$ is odd. Studying again the degrees of
the functions $h_j$ we get that
\[
\deg(h_{n-2k})=3k\quad\mbox{and}\quad \deg(h_{n-(2k+1)})\le 3k+1
,\quad k=0,1,\ldots (n-1)/2.
\]
In particular, $\deg(h_0)\le (3n-1)/2$ and $\deg(h_1)=3(n-1)/2.$
Again, as in the case $c_1\ne0$, the higher order terms in $x$
corresponding to equation \eqref{edo0} can not cancel. Therefore $n$
must be even, as we wanted to prove.
\end{proof}

\begin{proposition}\label{pro:h0h1} Let
\begin{equation*}
f_n(x,y)=h_n(x)y^n+h_{n-1}(x)y^{n-1}+\cdots+h_1(x)y+h_0(x)=0
\end{equation*}
be an irreducible invariant algebraic curve of system \eqref{main2}
with even degree, $n=2m$. Then
\begin{align}
h_0(x)&=\Big(\dfrac23\Big)^{m} x^{3m}+O\big(x^{3m-1}\big),\label{h0}\\
h_1(x)&=\dfrac15\Big(\dfrac23\Big)^{m}\left(5c_0-\big(5c_0+6mc\big)\frac{\left(\frac56\right)^{[m]}}{
 \left(\frac13\right)^{[m]}}\right) x^{3m-2}+O\big(x^{3m-3}\big),\label{h1}
\end{align}
where $x^{[m]}=x(x+1)(x+2)\cdots (x+m-1)$.
\end{proposition}

\begin{proof}
We proceed as in the proof of Proposition \ref{pro:odd}. The
coefficients $h_j$ of $f_n$ must satisfy the differential equations
\eqref{edos}, with $c_1=0.$ Arguing as in that proof we obtain the
degrees of each $h_j$. We can write
\begin{equation*}\label{aj}
h_j(x)=a_j(2m) x^{\deg(h_j)}+O\big(x^{\deg(h_j)-1}\big),
\end{equation*}
where,
\[
 \deg(h_j)=\begin{cases} 3k-2 ,\mbox{ when } j=2m-(2k-1),\\
 3k ,\mbox{ when } j=2m-2k,\end{cases}
\]
 for $k=0,1,\ldots, m$ and $a_{2m}(2m)=1.$ Let us determine these functions.

Plugging the above expressions in \eqref{edos} we obtain that the
terms $a_j=a_j(2m)$ satisfy the following recurrences
\begin{align}
a_{2m-2k}&=\dfrac{2m-(2k-2)}{3k} a_{2m-(2k-2)},\quad k=1,2,\ldots,m,\label{receven}\\
a_{2m-(2k+1)}&=\dfrac{(2m-(2k-1))a_{2m-(2k-1)}+h(2m-2k)a_{2m-2k}}{3k+1},\,
k=1,2,\ldots,m-1,\label{recodd}
\end{align}
where $h(j)=-(c_0+jc)$ and the initial conditions are
\[
a_{2m}=1\quad\mbox{and}\quad a_{2m-1}=h(2m)=-(c_0+2mc).
\]
The even terms $a_{2j}$ can be easily obtained from \eqref{receven}.
We get
\begin{equation}\label{even}
a_{2m-2j}={m \choose j}\Big(\dfrac23\Big)^{j}
\end{equation}
and in particular $a_0=(2/3)^m$ as we wanted to prove.  It remains
to  obtain the general expression of the last odd term
$a_1=a_1(2m)$. We take advantage of the linearity of the problem
with respect to the initial condition $a_{2m-1}$ and  write
\[
a_1(2m)=-\big(\widehat a_1(2m)c_0+\widetilde a_1(2m) c\big),
\]
where $\widehat a_1$ and $\widetilde a_1$ are the solution of the
recurrences  \eqref{receven}-\eqref{recodd} with  initial conditions
$a_{2m}=1$ and
\[
a_{2m-1}=1\quad\mbox{or}\quad a_{2m-1}=2m,
\]
respectively.

 Substituting  expression \eqref{even} in
\eqref{recodd} and developing the recurrent expressions we arrive at
\begin{align*}
\widehat a_1(2m)&=\sum_{j=0}^{m-1}{m\choose
j}\Big(\dfrac23\Big)^{j}\prod_{k=0}^{m-j-1}(2k+1)\prod_{k=j}^{m-1}\dfrac1{3k+1},
\\
\widetilde a_1(2m)&=2\sum_{j=0}^{m}(m-j){m\choose
j}\Big(\dfrac23\Big)^{j}\prod_{k=0}^{m-j-1}(2k+1)\prod_{k=j}^{m-1}\dfrac1{3k+1}.
\end{align*}
 We  introduce the following  auxiliary
functions
\begin{align*}
\alpha(m)&=\Gamma\Big(\frac12\Big)\Gamma\Big(\frac13+m\Big),\quad
\beta(m)=\dfrac{\Gamma(\frac12)\Gamma(\frac13)\Gamma(\frac56+m)}{\Gamma(\frac56)},\\
\gamma(m)&=\dfrac{\beta(m)}{\alpha(m)}=
\frac{\Gamma(\frac13)\Gamma(\frac56+m)}{\Gamma(\frac56)\Gamma(\frac13+m)}=\frac{\left(\frac56\right)^{[m]}}{
 \left(\frac13\right)^{[m]}}\,,
\end{align*}
where in the last equality we have used (v) of Lemma~\ref{gamma}.
Let us simplify the expressions of $\widehat a_1$ and $\widetilde
a_1$ using the above functions and the other equalities given in
Lemma~\ref{gamma}.

\begin{align*}
\widehat a_1(2m)&=\Big(\dfrac23\Big)^{m}\sum_{j=0}^{m-1}{m\choose
j}\prod_{k=0}^{m-j-1}\Big(\frac12+k\Big)\prod_{k=j}^{m-1}\dfrac1{\frac13+k}
\\&=\dfrac1{\alpha(m)}\Big(\dfrac23\Big)^{m}\sum_{j=0}^{m-1}{m\choose
j}\Gamma\Big(\frac12+m-j\Big)\Gamma\Big(\frac13+j\Big)\\
&=\dfrac1{\alpha(m)}\Big(\dfrac23\Big)^{m}\left(\sum_{j=0}^{m}{m\choose
j}\Gamma\Big(\frac12+m-j\Big)\Gamma\Big(\frac13+j\Big)-\Gamma\Big(\frac12\Big)\Gamma\Big(\frac13+m\Big)\right)\\
&=\dfrac1{\alpha(m)}\Big(\dfrac23\Big)^{m}\left(
\dfrac{\Gamma(\frac12)\Gamma(\frac13)\Gamma(\frac56+m)}{\Gamma(\frac56)}
-\alpha(m)\right)\\
&=\dfrac1{\alpha(m)}\Big(\dfrac23\Big)^{m}\left(\beta(m)-\alpha(m)\right)=
\Big(\dfrac23\Big)^{m}\left(\gamma(m)-1\right).
\end{align*}
Similarly,

\begin{align*}
\widetilde
a_1(2m)&=\dfrac2{\alpha(m)}\Big(\dfrac23\Big)^{m}\sum_{j=0}^{m}(m-j){m\choose
j}\Gamma\Big(\frac12+m-j\Big)\Gamma\Big(\frac13+j\Big)\\
&=\dfrac2{\alpha(m)}\Big(\dfrac23\Big)^{m}
\dfrac{\frac12\Gamma(\frac12)\Gamma(\frac13)\Gamma(\frac56+m)}{\frac56\Gamma(\frac56)}m\\
&=\dfrac1{\alpha(m)}\Big(\dfrac23\Big)^{m}\frac65\beta(m)m=\Big(\dfrac23\Big)^{m}\frac65\gamma(m)m .\\
\end{align*}
Hence
\begin{align*}
a_1(2m)&=-\big(\widehat a_1(2m)c_0+\widetilde
a_1(2m)c\big)\\&=-\Big(\dfrac23\Big)^{m}\left((\gamma(m)-1)c_0+\frac65\gamma(m)mc
\right)
\\&=\dfrac15\Big(\dfrac23\Big)^{m}\Big(5c_0-\big(5c_0+6mc\big)
\gamma(m)\Big)
\\
&=\dfrac15\Big(\dfrac23\Big)^{m}\left(5c_0-\big(5c_0+6mc\big)\frac{\left(\frac56\right)^{[m]}}{
 \left(\frac13\right)^{[m]}}\right),
\end{align*}
as we wanted to prove.
\end{proof}

When an invariant algebraic curve passes by an elementary critical
point, in many cases, the value of the cofactor at this point can be
obtained. These type of results, based on previous works of
Seidenberg (\cite{se}), are proved in~\cite{maite}.  In the next
proposition, which is included in \cite[Thm 14]{maite}, we  state
one of these cases.

\begin{proposition}\label{cof} Let $f(x,y)=0$ be an invariant algebraic
curve of a planar system  with corresponding cofactor $k(x,y)$.
Assume that it  contains a critical point of the system,
$(x_0,y_0)$, and that it is  a hyperbolic saddle with eigenvalues
$\lambda^-<0<\lambda^+.$ Then $k(x_0,y_0)\in\{
\lambda^+,\lambda^-,\lambda^++\lambda^- \}.$
\end{proposition}

\begin{proof}[Proof of Theorem~\ref{teo:iac}]
  By Propositions~\ref{pro:odd} and~\ref{pro:h0h1}  we know
that the invariant curve  has  even degree $n=2m, m\in\N$, and it
can be written as
\begin{equation*}
f(x,y)=h_n(x)y^n+h_{n-1}(x)y^{n-1}+\cdots+h_1(x)y+h_0(x)=0,
\end{equation*}
where $h_0$ and $h_1$ satisfy \eqref{h0} and \eqref{h1}. Moreover
its cofactor is constant,  $k(x,y)=c_0.$ Using that $h_0$ and $h_1$
must satisfy~\eqref{edo0} we get the identity
\begin{equation*}
-c_0h_0(x)-xh_1(x)+x^2h_1(x)\equiv0.
\end{equation*}
Using Proposition~\ref{pro:h0h1} we obtain that
\begin{equation*}
-c_0h_0(x)-xh_1(x)+x^2h_1(x)=-\frac{5c_0+6mc}5\Big(\dfrac23\Big)^{m}\frac{\left(\frac56\right)^{[m]}}{
 \left(\frac13\right)^{[m]}}x^{3m}+ O\big(x^{3m-1}\big).
\end{equation*}
Therefore, \begin{equation}\label{condii} 5c_0+6mc=0.
\end{equation}

The origin of \eqref{main2} is a saddle point with eigenvalues
$\lambda^\pm=\dfrac{-c\pm\sqrt{c^2+4}}2,$  where
$\lambda^-<0<\lambda^+$. Since, by hypothesis, $f(0,0)=0$ we can
apply Proposition~\ref{cof} to determine $c_0=k(0,0).$ We obtain
that $c_0\in\{\lambda^+,\lambda^-,-c\}.$ When $c_0=-c$,
equation~\eqref{condii} gives  $(6m-5)\,c=0$, which is in
contradiction with the hypothesis $c\ge2.$ Therefore, if the system
has an algebraic invariant curve under the above hypotheses, then
$c_0\in\{\lambda^+,\lambda^-\}.$ Take $c_0=\lambda^\pm$. Hence,
equation~\eqref{condii} writes as $6mc+5\lambda^\pm=0,$ or
equivalently,
\[
c=\mp\frac5{\sqrt{6}}\frac1{\sqrt{m(6m-5)}}.
\]
Imposing that $c\ge2$ we get that the only possibility is
$c_0=\lambda^-$ and $m=1$. Then, $c=5/\sqrt{6}$ as we wanted to
prove. Finally, simple computations give~\eqref{iac} and the theorem
follows.
\end{proof}

\subsection*{Acknowledgement} The first  author is partially supported by MINECO/FEDER grant number
MTM2008-03437 and Generalitat de Catalunya grant number
2009\-SGR410.


\begin{thebibliography}{10}
\bibitem{AblCla} M. J. Ablowitz and P.A. Clarkson, \emph{Solitons, Nonlinear Evolution Equations and Inverse Scattering},
 London Mathematical Society Lectures Notes \textbf{149}, Cambridge University Press (1991).

\bibitem{AZ} M. J. Ablowitz and A. Zeppetella, \emph{Explicit solutions
of Fisher's equation for a special wave speed}, Bull. Math. Biol.
\textbf{41} (1979), 835--840.

\bibitem{AA} N. Akhmediev and A. Ankiewicz ed., \emph{Dissipative Solitons: from Optics to Biology and Medicine},
 Lecture Notes in Physics \textbf{751}, Springer-Verlag, Berlin (2008).

\bibitem{maite} J. Chavarriga, H. Giacomini and M. Grau,
\emph{Necessary conditions for the existence of invariant algebraic
curves for planar polynomial systems}, Bull. Sci. Math. \textbf{129}
(2005), 99--126.

\bibitem{DraJoh} P. G. Drazin and R. S. Johnson, \emph{Solitons: An Introduction}, Cambridge University Press (1989).

\bibitem{DuLlAr} F. Dumortier, J. Llibre and J. C. Artes,
\emph{Qualitative Theory of Planar Differential Systems},
Universitext. Springer-Verlag, Berlin (2006).

\bibitem{Fisher} R. A. Fisher, \emph{The wave of advance of advantageous
genes}, Ann. Eugenics \textbf{7} (1937), 355--369.

\bibitem{GGG}  A. Gasull, H. Giacomini and M. Grau, \emph{On the stability of periodic orbits
for differential systems in $\R^n$}, Discrete Contin. Dyn. Syst.
Ser. B \textbf{10} (2008),  495--509.

\bibitem{GGT} A. Gasull, H. Giacomini and J. Torregrosa, \emph{Explicit upper and lower bounds
for the traveling wave solutions of Fisher-Kolmogorov type
equations}, Discrete Contin. Dyn. Syst. \textbf{33} (2013),
3567--3582.

\bibitem{GilKers} B. H. Gilding and R. Kersner, \emph{Travelling Waves in Nonlinear Diffusion Convection Reaction}, Birkh\~{A}¤user (2004).

\bibitem{Gor} A. Goriely, \emph{Integrability and Nonintegrability of Dynamical Systems}, Advanced Series in Nonlinear Dynamics,
\textbf{19}, World Scientific Publishing Co., Inc., River Edge, NJ (2001).

\bibitem{GriSchi} G. W. Griffiths and W. E. Schiesser, \emph{Traveling Wave Solutions of Partial Differential Equations:
 Numerical and Analytical Methods with Matlab and Maple}, Academic Press (2011).

\bibitem{Gri} P. Grindrod, \emph{Patterns and Waves, The Theory and Applications
of Reaction-Diffusion Equations}, Clarendon Press (1991).

\bibitem{J} J. H. He and X. H. Wu,
\emph{Exp-function method for nonlinear wave equations}, Chaos
Solitons and Fractals \textbf{30} (2006),  700--708.

\bibitem{InfRow} E. Infield and G. Rowlands, \emph{Nonlinear Waves, Solitons and Chaos}, 2nd Edition,
Cambridge University Press (2000).

\bibitem{Johnn} R. S. Johnson, \emph{A Modern Introduction to the Mathematical Theory of Water Waves},
 Cambridge University Press (1999).

\bibitem{KanNish} T. Kano and T. Nishida, \emph{A mathematical justification for Korteweg-de Vries equation and
Boussinesq equation of water surface waves}, Osaka J. Math. \textbf{23} (1986), 389--413.

\bibitem{knob} R. A. Knobel, \emph{An Introduction to the Mathematical Theory of Waves},
American Mathematical Society (2000).

\bibitem{Kolmo} A. Kolmogorov, I. Petrovskii and N. Piskunov,
\emph{A study of the diffusion equation with increase in the amount
of substance, and its application to a biological problem}, In V. M.
Tikhomirov, editor, \emph{Selected Works of A. N. Kolmogorov I},
248--270. Kluwer 1991. Translated by V. M. Volosov from Bull. Moscow
Univ., Math. Mech. \textbf{1} (1937), 1--25.

\bibitem{Kund} A. Kundu Ed., \emph{Tsunami and Nonlinear Waves}, Springer (2007).

\bibitem{Lieh} A. W. Liehr,\emph{ Dissipative Solitons in Reaction Diffusion Systems, Mechanism, Dynamics, Interaction},
Volume \textbf{70} of Springer Series in Synergetics, Springer, Berlin-Heidelberg (2013).

\bibitem{llibre} J. Llibre, \emph{Integrability of Polynomial Differential
Systems}, Handbook of Differential Equations, 437--532,
Elsevier/North-Holland, Amsterdam  (2004).

\bibitem{llibrezhang} J. Llibre and X. Zhang, \emph{Invariant algebraic surfaces of the
Lorenz system}, J. Math. Phys. \textbf{43} (2002),  1622--1645.

\bibitem{Maf} W. Mafliet, \emph{Solitary waves solutions of nonlinear wave equations},
 Am. J. Physics \textbf{60} (1992), 650--654.

\bibitem{MafHer1} W. Mafliet and W. Hereman, \emph{The tanh method I-exact solutions
 of nonlinear evolution wave equations},
 Physica Scripta \textbf{54} (1996), 563--568.

\bibitem{MafHer2} W. Mafliet and W. Hereman, \emph{The tanh method II-exact solutions
 of nonlinear evolution wave equations},
 Physica Scripta \textbf{54} (1996), 569--575.

\bibitem{Mur1}  J. D. Murray, \emph{Mathematical Biology. I. An Introduction}, Third edition.
 Interdisciplinary Applied Mathematics, \textbf{17}. Springer-Verlag, New York (2002).

\bibitem{Mur} J. D. Murray, \emph{Mathematical Biology. II. Spatial Models and Biomedical Applications},
Third edition, Interdisciplinary Applied Mathematics, \textbf{18},
Springer-Verlag, New York (2003).

\bibitem{NewWhi} A. C. Newell and J. A. Whitehead,
\emph{Finite bandwidth, finite amplitude convection}, J. Fluid Mech.
\textbf{38} (1969), 279--303.

\bibitem{PoZai} A. D. Polyanin and V. F. Zaitsev, \emph{Handbook of Nonlinear Partial Differential Equations}, Chapman and Hall/CRC, Boca
Raton (2004).

\bibitem{RodMiu} M. R. Rodrigo and R. M. Miura, \emph{Exact and approximate
traveling waves of reaction-diffusion systems via a variational
approach}, Anal. Appl. (Singap.) \textbf{9} (2011), 187--199.

\bibitem{SanMai1995} F. S\'anchez-Gardu\~no and P. K.
Maini, \emph{Travelling wave phenomena in some degenerate
reaction-diffusion equations}, J. Differential Equations
\textbf{117} (1995), 281--319.

\bibitem{SanMai1997} F. S\'anchez-Gardu\~no and P. K.
Maini, \emph{Travelling wave phenomena in non-linear diffusion
degenerate Nagumo equations}, J. Math. Biol. \textbf{35} (1997),
713--728.

\bibitem{seg} L. A. Segel, \emph{Distant sidewalls cause slow amplitude modulation of cellular convection},
J. Fluid Mech. \textbf{38} (1969), 203-224.

\bibitem{se} A. Seidenberg,
\emph{ Reduction of singularities of the differential equation $Ady
= B dx$}, Amer. J. Math. \textbf{90} (1968), 248--269.

\bibitem{St} B. Sturmfels,  \emph{Solving Systems of Polynomial Equations}. CBMS
Regional Conference Series in Mathematics, \textbf{97}, Published for the
Conference Board of the Mathematical Sciences, Washington, DC; by
the American Mathematical Society, Providence, RI, 2002. viii+152

\bibitem{Xin} J. Xin, \emph{Front propagation in
heterogeneous media}, SIAM Rev. \textbf{42} (2000), 161--230.

\bibitem{ZelFra} Y. B. Zeldovich and D. A. Frank-Kamenetskii,
\emph{A theory of thermal propagation of flame}, Acta Physicochimica
URSS \textbf{9} (1938), 341--350. English translation: In Dynamics
of curved fronts, editor R. Pelc\'e, Perspectives in Physics Series,
Academic Press, New York (1988), 131--140.

\bibitem{ZonHon} Z. Yang and B. Y.C. Hon, \emph{An improved modified extended tanh function method},
 Z. Naturforsch \textbf{61} (2006), 103--115.


\end{thebibliography}
\end{document}